\documentclass[11pt,a4paper]{article}
\usepackage{amsmath,amsfonts,amssymb,amsthm}
\title{Random points in halfspheres\footnote{Partially supported by ERC Advanced Research Grant no 267165 (DISCONV). The first author was supported by Hungarian National Foundation Grant K 111827.
The second author was partially supported by the German Research Foundation (DFG) under the grant HU 1874/4-2.}}

\author{Imre B\'{a}r\'{a}ny, Daniel Hug, Matthias Reitzner, Rolf Schneider}
\date{}
\usepackage{amssymb}
\usepackage[all,cmtip]{xy}
\usepackage{xcolor}

\newtheorem{Proposition}{Proposition}[section]

\newtheorem{Lemma}[Proposition]{Lemma}
\newtheorem{Theorem}[Proposition]{Theorem}

\newcommand{\Sd}{{\mathbb S}^d}
\newcommand{\Rd}{{\mathbb R}^{d+1}}

\newcommand{\D}{{\rm d}}
\newcommand{\fed}{\,\rule{.1mm}{.26cm}\rule{.24cm}{.1mm}\,}

\newcommand{\Pb}{{\mathbb P}}
\newcommand{\E}{{\mathbb E}\,}
\newcommand{\eps}{\varepsilon}

\newcommand{\Se}{{\mathbb S}_e^+}
\newcommand{\dS}{\partial\hspace{1pt}{\mathbb S}_e^+}

\begin{document}

\maketitle

\begin{abstract}
A random spherical polytope $P_n$ in a spherically convex set $K \subset S^d$ as considered here is the spherical convex hull of $n$ independent, uniformly distributed random points in $K$. The behaviour of $P_n$ for a spherically convex set $K$ contained in an open halfsphere is quite similar to that of a similarly generated random convex polytope in a Euclidean space, but the case when $K$ is a halfsphere is different. This is what we investigate here,  establishing the asymptotic behaviour, as $n$ tends to infinity, of the expectation of several characteristics of $P_n$, such as facet and vertex number, volume and surface area. For the Hausdorff distance from the halfsphere, we obtain also some almost sure asymptotic estimates.
\end{abstract}



\section{Introduction}\label{sec1}

Ever since the seminal articles by R\'{e}nyi and Sulanke \cite{RS63}, \cite{RS64}, convex hulls of random points and their asymptotic behaviour when the number of points increases, have been a favourite topic in stochastic geometry. For overviews, we refer the reader to the Notes for Subsection 8.2.4 in \cite{SW08} and to the more recent surveys \cite{Rei10} and \cite{Hug13}. In a thoroughly studied setting, one assumes $n$ stochastically independent, uniformly distributed points in a given convex body $K$ in $d$-dimensional Euclidean space and studies the asymptotic behaviour of the convex hull of the random points as the number $n$ tends to infinity. Quantities of interest may be the face numbers of the random polytopes or the quality of approximation of $K$ by the polytopes, measured, for instance, by differences of volumes or intrinsic volumes or by the Hausdorff distance. Already the first articles by R\'{e}nyi and Sulanke exhibited the strong influence of the boundary structure of $K$ on the asymptotic behaviour of the random polytopes. For smooth bodies and for polytopes, for example, the asymptotics are essentially different, and in the former case, curvatures enter the results in an essential way.

More recently, stochastic geometry in spherical spaces has found increasing interest. Many of the questions that have been treated in Euclidean space have counterparts for spherical space, and may have similar answers. A new phenomenon, however, arises if one considers random points in a halfsphere. Its boundary, as a submanifold, has zero curvature, and this should lead to types of asymptotic behaviour that cannot be observed in Euclidean spaces. It is the purpose of this note to collect first results on spherical convex hulls of random points in halfspheres. For $n$ independent, uniformly distributed random points in a $d$-dimensional closed halfsphere, we study their spherical convex hull and investigate the expected values of some geometric functionals for these spherical polytopes. For facet number, surface area, and spherical mean width, we obtain explicit integral expressions for their expectations, from which the asymptotic behaviour, as $n\to\infty$, can be deduced. For volume and vertex number, some more elaborate arguments are required to determine their asymptotics. Finally, we establish bounds for the almost sure asymptotic behaviour of the spherical Hausdorff distance between the considered random polytopes and the halfsphere.

\section{Preliminaries}\label{sec2}

The sphere $\Sd$ ($d\ge 2$) in which we are interested is taken as the unit sphere of the real vector space $\Rd$ with its standard scalar product $\langle\cdot\, ,\cdot\rangle$. The spherical distance of two points $x,y\in\Sd$ is given by $d_s(x,y)=\arccos\langle x,y\rangle$. The spherical Hausdorff distance of two nonempty compact sets $K,M\subset\Sd$ is defined by
\begin{equation}\label{2.0}
\delta_s(K,M):=\max\left\{\max_{x\in K}\min_{y\in M} d_s(x,y),\, \max_{x\in M} \min_{y\in K} d_s(x,y)\right\}.
\end{equation}

We denote by $\lambda$ the Lebesgue measure on $\Rd$ and by $\sigma$ the spherical Lebesgue measure on $\Sd$. The $(d-1)$-dimensional spherical Lebesgue measure on great subspheres of dimension $d-1$ is denoted by $\sigma_{d-1}$. The constant
$$ \omega_{d+1}= \sigma(\Sd) = \frac{2\pi^{\frac{d+1}{2}}}{\Gamma\left(\frac{d+1}{2}\right)}$$
is the total measure of $\Sd$.
The number $\kappa_{d+1}=\omega_{d+1}/(d+1)$ is the volume of the unit ball in $\Rd$.

By $G(d+1,d)$ we denote the Grassmannian of $d$-dimensional linear subspaces of $\Rd$, equipped with its standard topology. The unique rotation invariant (Borel) probability measure on $G(d+1,d)$ is denoted by $\nu$.

We fix a vector $e\in \Sd$ and consider the closed halfsphere
$$ \Se:=\{u\in \Sd: \langle u,e\rangle\ge 0\}.$$
Its boundary is denoted by $\dS:=e^\perp\cap\Sd$. The uniform probability measure $\mu$ on $\Se$ is given by
$$ \mu= \frac{2\sigma\fed\Se}{\omega_{d+1}}.$$

For $x\in\Se$ and $\eps>0$, we define by $B(x,\eps):= \{y\in\Se:d_s(x,y)\le\eps\}$ the closed ball in $\Se$ with centre $x$ and radius $\eps$. 

By a {\em spherical polytope} in $\Sd$ we understand here the intersection $C\cap \Sd$ of $\Sd$ with a pointed closed convex polyhedral cone $C$ in $\Rd$. Such a cone is the intersection of finitely many closed halfspaces with $0$ in the boundary, provided that it is different from $\{0\}$ and does not contain a line. The set of spherical polytopes is equipped with the topology induced by the spherical Hausdorff distance and the corresponding Borel structure. For a spherical polytope $P$ and for $k\in\{0,\dots,d-1\}$, we understand by ${\mathcal F}_k(P)$ the set and by $f_k(P)$ the number of its $k$-dimensional faces.

We consider stochastically independent random points $X_1,\dots,X_n$, $n\ge d+1$, in $\Se$, each with distribution $\mu$. Then we define the spherical random polytope
\begin{equation}\label{2.1}
P_n:= {\rm conv}_s\{X_1,\dots,X_n\},
\end{equation}
where the spherical convex hull of a set $A\subset \Sd$ is defined by
$$ {\rm conv}_s(A):= \Sd\cap{\rm pos}\,A$$
and ${\rm pos}\,$ denotes the positive hull in $\Rd$.

\section{Facet Number}\label{sec3}

In this section and in Section 5, we consider functionals of spherical polytopes which are of the following type. Let $\eta$ be a rotation invariant nonnegative measurable function on spherical $(d-1)$-polytopes. For a spherical polytope $P$, let
$$ \varphi(\eta,P):= \sum_{F\in{\mathcal F}_{d-1}(P)} \eta(F).$$
Thus, for $\eta(F)=1$, we get the facet number $\varphi(1,P)=f_{d-1}(P)$, and $\eta(F)=\sigma_{d-1}(F)$ yields the  surface area $\varphi(\sigma_{d-1},P)=:S(P)$.

We will be interested in the random variable
$$ \varphi(\eta,n) := \varphi(\eta,P_n)$$
with $P_n$ given by (\ref{2.1}). In particular, $\varphi(1,n)=f_{d-1}(P_n)$ and $\varphi(\sigma_{d-1},n)=S(P_n)$.

\begin{Theorem}\label{T1} Let $P_n$ be the spherical convex hull of $n\ge d+1$ independent uniform random points on the halfsphere $\Se$. Then
\begin{equation}\label{3.1}
\E f_{d-1}(P_n) =  \frac{2\omega_d}{\omega_{d+1}} \binom{n}{d}\int_0^\pi \left(1-\frac{\alpha}{\pi}\right)^{n-d}\sin^{d-1}\alpha\,\D \alpha.
\end{equation}
Further,
\begin{equation}\label{3.2}
\lim_{n\to\infty} \E f_{d-1}(P_n)= 2^{-d}d!\kappa_d^2.
\end{equation}
\end{Theorem}

Since $P_n$ is almost surely a simplicial polytope, it satisfies the Dehn--Sommerville equation
$$ 2f_{d-2}(P_n)=df_{d-1}(P_n)$$
(see \cite{Gru03}, p. 146). Therefore, Theorem \ref{T1} also immediately yields the expectation $\E f_{d-2}(P_n)$.

The integral in (\ref{3.1}) can in principle be evaluated by using recursion formulas and known definite integrals; e.g., see \cite[ p. 117]{GH50}. (The evaluation of the integral for $d=2$ in \cite[(6.16)]{Mil71}, is corrected in \cite{CM09}.)

The fact that the expectation $\E f_{d-1}(P_n)$ has an explicit expression for each $n$ and not only an asymptotic expression for $n\to\infty$, is one of the new phenomena not observed in the Euclidean case. We emphasize also the finiteness of the limit (\ref{3.2}). We remark that the image measure of $\mu$, restricted to the interior of $\Se$, under the gnomonic projection $v\mapsto \langle e,v\rangle^{-1}v-e$, yields a probability distribution on $e^\perp$, identified with ${\mathbb R}^d$, with the following property. It is rotationally symmetric, and the convex hull of $n$ independent random points with this distribution has a facet number whose expectation has a finite limit, for $n\to \infty$. For $d=2$, distributions with these properties were first constructed by Carnal \cite{Car70}. Our approach provides natural examples to this effect, also in higher dimensions.

For the proof of Theorem \ref{T1}, we need a Blaschke--Petkantschin formula on the sphere. Very general formulas of this type were proved by Arbeiter and Z\"ahle \cite{AZ91}. The simple case needed here follows immediately from the linear Blaschke--Petkantschin formula in Euclidean space, as we briefly indicate.

\begin{Lemma}\label{L3.1} Let $f:(\Sd)^d \to {\mathbb R}$ be nonnegative and measurable. Then
$$ \int_{(\Sd)^d} f\,\D\sigma^d = \frac{\omega_{d+1}}{2} \int_{G(d+1,d)} \int_{(H\cap\hspace{1pt}\Sd)^d} f\nabla_d\,\D\sigma_{d-1}^d \,\nu(\D H).$$
Here $\nabla_d(x_1,\dots,x_d)$ denotes the $d$-dimensional volume of the parallelepiped spanned by the vectors $x_1,\dots,x_d\in\Rd$, and $\sigma_{d-1}$ is the $(d-1)$-dimensional spherical Lebesgue measure on $H\cap \Sd$.
\end{Lemma}

\begin{proof} We choose a measurable function $g:[0,\infty)\to {\mathbb R}$ with
$$ \int_0^\infty g(r)r^d\,\D r=1$$
and define $F:({\mathbb R}^{d+1})^d\to{\mathbb R}$ by
$$ F(r_1u_1,\dots,r_du_d) := g(r_1)\cdots g(r_d)F(u_1,\dots,u_d),\quad r_i\ge 0,\,u_i\in \Sd.$$
Applying the linear Blaschke--Petkantschin formula (\cite{SW08}, Thm. 7.2.1) to $F$, we get
\begin{equation}\label{3.3}
\int_{({\mathbb R}^{d+1})^d} F\,\D\lambda^d = \frac{\omega_{d+1}}{2} \int_{G(d+1,d)} \int_{H^d} F\nabla_d\,\D \lambda_{d-1}^d\,\nu(\D H),
\end{equation}
where $\lambda_{d-1}$ denotes the $d$-dimensional Lebesgue measure on $H$. Using polar coordinates to transform the integrals over $({\mathbb R}^{d+1})^d$ and $H^d$, we obtain the statement of the lemma.
\end{proof}

We need Lemma~\ref{L3.1}  for a function $f$ defined on $\Se$. We extend this function to $\Sd$ by putting $f(u_1,\dots,u_d):=0$ if one of the arguments is in $\Sd\setminus\Se$, then
\begin{equation}\label{3.4}
\int_{(\Se)^d} f\,\D\mu^d= \left(\frac{2}{\omega_{d+1}}\right)^{d-1} \int_{G(d+1,d)} \int_{(H\cap\hspace{1pt}\Se)^d} f\nabla_d\,\D\sigma_{d-1}^d \,\nu(\D H).
\end{equation}

A first formula for the expectation of $\varphi(\eta,n)$ can be obtained similarly as in the Euclidean case (see, e.g., \cite[p. 319]{SW08}). For this, let $u_1,\dots,u_d\in\Se$ be linearly independent, and let $H$ be the $d$-dimensional linear subspace spanned by these vectors. Denoting by $H^+,H^-$ the two closed halfspaces bounded by $H$, we define
$$ H^+(u_1,\dots,u_d):= H^+\cap \Se,\qquad  H^-(u_1,\dots,u_d):= H^-\cap \Se.$$
Which of the halfspaces is denoted by $H^+$ is irrelevant, since we consider only symmetric functions of $H^+$ and $H^-$. As in the Euclidean case, one shows that
\begin{align*}
{\mathbb E}\,\varphi(\eta,n)&=
\binom{n}{d} \int_{(\Se)^d} \left[\mu(H^+(u_1,\dots,u_d))^{n-d}+\mu(H^-(u_1,\dots,u_d))^{n-d}\right]\\
&\hspace{5mm}\times\;\eta({\rm conv}_s\{u_1,\dots,u_d\})\,\mu^d(\D(u_1,\dots,u_d)).
\end{align*}
An application of (\ref{3.4}) yields
\begin{align}
& {\mathbb E}\,\varphi(\eta,n)\nonumber\\
&=\left(\frac{2}{\omega_{d+1}}\right)^{d-1} \binom{n}{d} \int_{G(d+1,d)} \int_{(H\cap\hspace{1pt}\Se)^d} \left[\mu(H^+\cap\Se)^{n-d}+ \mu(H^-\cap\Se)^{n-d}\right]\nonumber\\
&\hspace{5mm}\times\;\eta({\rm conv}_s\{u_1,\dots,u_d\})\nabla_d(u_1,\dots,u_d)\,\sigma_{d-1}^d(\D(u_1,\dots,u_d)) \,\nu(\D H)\label{3.6neu}\\
&= \left(\frac{2}{\omega_{d+1}}\right)^{d-1} C(\eta,d) \binom{n}{d} \int_{G(d+1,d)} \left[\mu(H^+\cap\Se)^{n-d}+
\mu(H^-\cap\Se)^{n-d}\right]\nu(\D H)\nonumber
\end{align}
with
$$ C(\eta,d)= \int_{(H\cap\hspace{1pt}\Se)^d}\eta({\rm conv}_s\{u_1,\dots,u_d\})\nabla_d(u_1,\dots,u_d)\, \sigma_{d-1}^d(\D(u_1,\dots,u_d)),$$
which is independent of $H\in G(d+1,d)\setminus \{e^{\perp}\}$.
Here we have made use of the assumption that the function $\eta$ is rotation invariant.
For $v\in \Sd$, let
$$ v^+:=\{u\in \Sd:\langle u,v\rangle\ge 0\},\quad v^-:=\{u\in \Sd:\langle u,v\rangle\le 0\}.$$
Then
\begin{align*}
& \int_{G(d+1,d)} \left[\mu(H^+\cap\Se)^{n-d}+\mu(H^-\cap\Se)^{n-d}\right]\nu(\D H)\\
&= \frac{1}{\omega_{d+1}} \int_{\Sd} \left[\mu(v^+\cap\Se)^{n-d}+\mu(v^-\cap\Se)^{n-d}\right]\sigma(\D v)\\
&= \frac{2}{\omega_{d+1}} \int_{\Sd} \left[1-\mu(v^-\cap\Se)\right]^{n-d}\sigma(\D v).
\end{align*}
This gives
$$ {\mathbb E}\,\varphi(\eta,n)=  \left(\frac{2}{\omega_{d+1}}\right)^d C(\eta,d) \binom{n}{d}\int_{\Sd} \left[1-\mu(v^-\cap\Se)\right]^{n-d} \,\sigma(\D v).$$
For  $v\in \Sd$, we write
$$ v=(\cos\alpha)e+(\sin\alpha)\overline v \qquad  \text{with }\overline v\in \dS,\;\alpha\in [0,\pi].$$
Then
$$ \mu(v^-\cap \Se)=\frac{\alpha}{\pi}$$
and
\begin{equation} \label{3.5}
{\mathbb E}\,\varphi(\eta,n) = \left(\frac{2}{\omega_{d+1}}\right)^d C(\eta,d)\, \omega_d \binom{n}{d}\int_0^\pi \left(1-\frac{\alpha}{\pi}\right)^{n-d}\sin^{d-1}\alpha\,\D\alpha.
\end{equation}

Now we specialize $\eta$. For $\eta=1$, we choose $f=1$ in (\ref{3.4}) and get
$$  C(1,d)= \left(\frac{\omega_{d+1}}{2}\right)^{d-1}.$$
Together with (\ref{3.5}), this yields the assertion (\ref{3.1}).

The limit relation (\ref{3.2}) is obtained from the asymptotic expansion of Lemma \ref{L2}, which we prove in the next section. To obtain the right-hand side of (\ref{3.2}), we use the Legendre duplication formula, to get
\begin{align*}
\frac{2\omega_d}{\omega_{d+1}}\frac{\pi^d}{d}
&=2\frac{\pi^{\frac{d}{2}}}{\pi^{\frac{d+1}{2}}}\frac{\Gamma \left(\frac{d+1}{2}\right)}{\Gamma\left(\frac{d}{2}\right)}\frac{\pi^d}{d}
=\frac{1}{\sqrt{\pi}}\frac{\Gamma\left(\frac{d+1}{2}\right)}{\Gamma\left(\frac{d}{2}+1\right)}\pi^d\\
&=\frac{1}{\sqrt{\pi}}\frac{1}{\Gamma\left(\frac{d}{2}+1\right)}
\frac{\sqrt{\pi}\,2^{-d}\Gamma(d+1)}{\Gamma\left(\frac{d}{2}+1\right)}\pi^d
=2^{-d}d!\kappa_d^2.
\end{align*}

\section{An Asymptotic Expansion}\label{sec5}

We need repeatedly the following asymptotic expansion.

\begin{Lemma}\label{L2}
$$ \binom nd \int_0^\pi \left(1- \frac \alpha \pi \right)^{n-d} \sin^{d-1} \alpha \, \D\alpha = \frac {\pi^d}d  \left[  1 - \binom {d+1} 3  \pi^2 n^{-2} + O(n^{-3}) \right]$$
as $n \to \infty$.
\end{Lemma}
\begin{proof}
The substitution
$1- \frac \alpha \pi = e^{-s}$, $\alpha = \pi(1-e^{-s})$, yields
$$ I:= \int_0^\pi \left(1- \frac \alpha \pi \right)^{n-d} \sin^{d-1} \alpha \, \D\alpha=
\pi \int_0^{\infty } e^{-s(n-d)-s} \sin^{d-1} \left( \pi (1-e^{-s})  \right) \, \D s.$$
We expand
$$
e^{(d-1)s}= 1 + (d-1)s + \frac{(d-1)^2}{2} s^2 +O(s^3)
$$
and
\begin{align*}
\sin^{d-1} \left( \pi (1-e^{-s})  \right)
&=
\left( \pi (1-e^{-s}) - \frac 16 \pi^3 (1-e^{-s})^3 + O((1-e^{-s})^5) \right)^{d-1}
\\ &=
\left( \pi \left(s- \frac 12 s^2 + \frac 16 s^3 + O(s^4)\right) - \frac 16 \pi^3 \left(s + O(s^2)\right)^3 + O(s^5) \right)^{d-1}
\\ &=
\left( \pi s- \frac \pi 2 s^2 + \frac {\pi - \pi^3} 6 s^3 + O(s^4) \right)^{d-1}
\\ &=
\pi^{d-1} s^{d-1} \left( 1 - \frac {d-1} 2 s +
\frac {(d-1)(3d-2 - 4 \pi^2)} {24} s^2 +
O(s^3) \right) .
\end{align*}
Multiplying both expansions gives
\begin{align*}
& e^{(d-1)s} \sin^{d-1} \left( \pi (1-e^{-s})  \right) \\
& =
\pi^{d-1} s^{d-1}
\left(
1 + \frac {d-1} 2 s
+  \frac {(d-1)(3d-2 - 4 \pi^2)} {24} s^2
+ O(s^3) \right) .
\end{align*}
We insert this into the integral $I$ and obtain
$$ I= \pi^d \int_0^{\infty } e^{-sn} s^{d-1} \left(1 + \frac {d-1} 2 s +  \frac {(d-1)(3d-2 - 4 \pi^2)} {24} s^2  + O(s^3) \right)
 \, ds .$$
Substituting $sn=t$ yields
\begin{align*}
I &= \pi^d n^{-d} \int_0^{\infty } e^{-t} t^{d-1} \left(  1 + \frac {d-1} 2 \frac 1n t +  \frac {(d-1)(3d-2 - 4 \pi^2)} {24} \frac 1{n^2} t^2  + O\left(\frac 1{n^3} t^3\right) \right)\, dt\\
&= \pi^d (d-1)! n^{-d} \left[1 + \binom d 2  n^{-1} +\frac {3d-2 - 4 \pi^2} {4} \binom {d+1}3  n^{-2}  +
O( n^{-3}) \right].
\end{align*}
In the last step we multiply this by the expansion
$$\binom n d = \frac 1{d!} n^d \left[1 - \binom d 2  n^{-1} + \frac{(3d-1)}{4} \binom {d}3 n^{-2} + O(n^{-3}) \right],$$
which leads to
$$
\binom{n}{d} I=\frac {\pi^d}d  \left[  1 - \binom {d+1} 3  \pi^2 n^{-2} + O(n^{-3}) \right],
$$
as stated.
\end{proof}

\section{Surface Area}\label{sec4}

If we choose $\eta=\sigma_{d-1}$ in (\ref{3.5}), we obtain
\begin{equation} \label{4.1}
\E S(P_n) = \left(\frac{2}{\omega_{d+1}}\right)^d C(\sigma_{d-1},d)\, \omega_d \binom{n}{d}\int_0^\pi \left(1-\frac{\alpha}{\pi}\right)^{n-d}\sin^{d-1}\alpha\,\D\alpha.
\end{equation}
Together with Lemma \ref{L2}, this yields
$$ \E S(P_n) =\left(\frac{2}{\omega_{d+1}}\right)^d C(\sigma_{d-1},d)\, \omega_d  \frac {\pi^d}d  \left[1 - \binom {d+1} 3  \pi^2 \, n^{-2} + O(n^{-3})
\right].$$
Since $ \E S(P_n)\to \sigma_{d-1}(\dS)= \omega_d$ for $n\to\infty$, as is easy to see (and in particular is implied by Theorem \ref{th:LDI-upperbound}), it follows that
$$ C(\sigma_{d-1},d) = \left(\frac{\omega_{d+1}}{2}\right)^d \frac{d}{\pi^d}.$$
Thus, we have obtained the following result.

\begin{Theorem}\label{T2} For $P_n$ as in Theorem \ref{T1},
\begin{equation}\label{4.2}
\E S(P_n) = \frac{d\omega_d}{\pi^d} \binom{n}{d}\int_0^\pi \left(1-\frac{x}{\pi}\right)^{n-d}\sin^{d-1}x\,\D x.
\end{equation}
Further,
\begin{equation}\label{4.3}
{\mathbb E} \,S(P_n) = \omega_d\left(1-\binom{d+1}{3}\pi^2\, n^{-2}+O(n^{-3})\right)
\end{equation}
as $n\to\infty$.
\end{Theorem}

\section{Spherical Mean Width}\label{sec6}

In Euclidean space, the surface area is one functional in the series of intrinsic volumes (or quermassintegrals, with a different normalization), which range from Euler characteristic and mean width to volume. All of these have been studied for random polytopes. In spherical space, the intrinsic volumes and quermassintegrals have counterparts which are different, though connected by linear relations. We consider here one of these functionals, the {\em spherical mean width} $U_1$. For a spherically convex body $K\subset \Sd$, it is defined by
$$ U_1(K):= \frac{1}{2} \int_{G(d+1,d)} \chi(K\cap H)\,\nu(\D H),$$
where $\chi$ denotes the Euler characteristic. The normalizing factor $1/2$ is convenient; for instance, $U_1(\Se)=1/2$. The definition of the spherical mean width is analogous to the integral representation of the Euclidean mean width. Also some of its properties are analogous; for example, an Urysohn inequality for the spherical mean width was proved in \cite{GHS02}.

\begin{Theorem}\label{T3}
For $P_n$ as in Theorem \ref{T1},
\begin{align}\label{6.1}
\E U_1(P_n) &= \frac{1}{2} - \frac{\omega_d}{\omega_{d+1}} \int_0^\pi\left(1-\frac{\alpha}{\pi}\right)^n \sin^{d-1} \alpha \,\D\alpha\\
&= \frac{1}{2} - \frac{\omega_d}{\omega_{d+1}} (d-1)!\pi^d \, n^{-d} +O(n^{-(d+2)}).\label{6.2}
\end{align}
\end{Theorem}

\vspace{2mm}

\begin{proof} We have
\begin{align*}
\E U_1(P_n) &= \E\frac{1}{2} \int_{G(d+1,d)}  {\bf 1} \{H\cap P_n\not=\emptyset\}\,\nu(\D H)\\
&= \frac{1}{2} \int_{G(d+1,d)} \left[1-\Pb(H\cap P_n=\emptyset)\right]\nu(\D H)\\
&= \frac{1}{2} \int_{G(d+1,d)} \left[1-\mu(H^+\cap \Se)^n-\mu(H^-\cap\Se)^n\right]\nu(\D H)\\
&= \frac{1}{2} - \frac{1}{2} \int_{G(d+1,d)} \left[\mu(H^+\cap \Se)^n+\mu(H^-\cap\Se)^n\right]\nu(\D H)\\
&= \frac{1}{2} - \frac{1}{\omega_{d+1}} \int_{\mathbb{S}^d} \left[1-\mu(v^-\cap\Se)\right]^n\sigma(\D v)\\
&= \frac{1}{2} - \frac{\omega_d}{\omega_{d+1}} \int_0^\pi\left(1-\frac{\alpha}{\pi}\right)^n\sin^{d-1}\alpha\,\D\alpha,
\end{align*}
which is (\ref{6.1}). By Lemma \ref{L2},
$$ \int_0^\pi\left(1-\frac{\alpha}{\pi}\right)^n\sin^{d-1}\alpha\,\D\alpha = \binom{n+d}{d}^{-1}\frac {\pi^d}d  
\left[  1 - \binom {d+1} 3  \pi^2\, n^{-2} + O(n^{-3}) \right]$$
which gives (\ref{6.2}).
\end{proof}

\section{Volume and Vertex Number}\label{sec7}

As before, we assume that $P_n={\rm conv}_s\{X_1,\dots,X_n\}$ with $n\ge d+1$ independent random points $X_1,\dots,X_n\in\Se$ with distribution $\mu$. It is clear that $\E \sigma(P_n)\to \sigma(\mathbb{S}_e^+)$ as $n\to\infty$. The following theorem shows that the speed of convergence is of the order $n^{-1}$, and hence different from the orders in the case of surface area or mean width approximation. 

\begin{Theorem}\label{T7.1} For $P_n$ as above,
\begin{equation}\label{7.1}
{\mathbb E}\,\sigma(\Se\setminus P_n)
= C(d) \,\pi^{d+1}\left(\frac{2}{\omega_{d+1}}\right)^{d}\omega_d  \, {n^{-1}}+O\left({n^{-2}}\right),
\end{equation}
where the constant $C(d)$ is defined by \eqref{boldeff}. Further,
\begin{equation}\label{7.2}
\lim_{n\to\infty}\E f_0(P_{n})=C(d) \,\pi^{d+1}\left(\frac{2}{\omega_{d+1}}\right)^{d+1}\omega_d.
\end{equation}
\end{Theorem}

\begin{proof}
We start with some preparations. For $z\in\Se\setminus\{e\}$, we denote by $\Pi(z)\in \dS$ the metric (or orthogonal) projection of $z$ to $\dS$, which is determined by $z=(\cos \alpha)\Pi(z)+(\sin \alpha) e$, for some $\alpha\in [0,\frac{\pi}{2})$. For a set $A\subset\Se\setminus\{e\}$, let $\Pi(A):=\{\Pi(a):a\in A\}$. 
Then, if $F\subset \Se$ is spherically convex and $e\notin F$, we have
$$
\Pi(F)=\text{conv}_s(F\cup\{e,-e\})\cap e^\perp.
$$
For any such $F$, we define 
$$
\eta_{\Delta}(F):=\sigma\left(\text{conv}_s(\Pi(F)\cup\{e\})\right)-\sigma\left(\text{conv}_s(F\cup\{e\})\right).
$$
If $P\subset\Se$ is a spherically convex polytope with $e\in\text{int}\, P$ and  $F\in\mathcal{F}_{d-1}(P)$, then 
$\eta_{\Delta}(F)$ is the volume `under $F$'. If $e\in\text{int}\, P$, we therefore get
$$
\sigma(\Se\setminus P)=\frac{1}{2}\omega_{d+1}-\sigma(P)=\varphi(\eta_\Delta,P).
$$

Let $r\in (0,\pi/2)$ be fixed. Then there is a constant $c\in (1/2,1)$ (without loss of generality), depending only on $d$ and $r$, such that 
\begin{align}\label{negli}
\mathbb{P}(B(e,r)\not\subset P_n)=O(c^n).
\end{align}
In fact, we can choose $m$ points $p_1,\dots,p_m\in\dS$ and a number $\rho>0$, where $m$ and $\rho$ depend only on $d$ and $r$, such that the balls $B(p_i,\rho)$ are pairwise disjoint and that $B(p_i,\rho)\cap \{X_1,\dots,X_n\}\not=\emptyset$ for $i=1,\dots,m$ implies $B(e,r)\subset P_n$. Then
$$
\Pb(B(e,r)\not\subset P_n) \le  \sum_{i=1}^m \Pb(B(p_i,\rho)\cap\{X_1,\dots,X_n\}=\emptyset) = m(1-\mu(B(p_1,\rho))^n,
$$
which gives (\ref{negli}). As a consequence, we may assume in the following that $\delta_s(P_n,\Se) < \pi/4$, adding an error term $O(c^n)$ where necessary. In particular, we can assume that $e\in \text{int}\, P_n$. 

For $H\in G(d+1,d)$ with $e\notin H$, we write $H^e$ for the uniquely determined halfspace bounded by $H$ which contains $e$. Similarly, we put $v^e:=(v^\perp)^e$ 
if $v$ is a unit vector.  By an obvious modification of the argument leading to \eqref{3.6neu}, we obtain
\begin{align}
& {\mathbb E}\,\sigma(\Se\setminus P_n)\nonumber\\
&=\left(\frac{2}{\omega_{d+1}}\right)^{d-1} \binom{n}{d} \int_{G(d+1,d)} \int_{(H\cap\hspace{1pt}\Se)^d}  \mu(H^e\cap\Se)^{n-d}\nonumber\\
&\hspace{5mm}\times\;\eta_\Delta({\rm conv}_s\{u_1,\dots,u_d\})\nabla_d(u_1,\dots,u_d)\,\sigma_{d-1}^d(\D(u_1,\dots,u_d))\,\nu(\D H)+O(c^{n})\nonumber\\
&=\left(\frac{2}{\omega_{d+1}}\right)^{d-1}\frac{1}{\omega_{d+1}} \binom{n}{d} \int_{\mathbb{S}^d} \int_{(v^\perp \cap\hspace{1pt}\Se)^d} \mu(v^e  \cap\Se)^{n-d}\nonumber\\
&\hspace{5mm}\times\;\eta_\Delta({\rm conv}_s\{u_1,\dots,u_d\})\nabla_d(u_1,\dots,u_d)\,\sigma_{d-1}^d(\D(u_1,\dots,u_d))\,\sigma(\D v)+O(c^{n})\nonumber\\
&=\left(\frac{2}{\omega_{d+1}}\right)^{d-1}\frac{2}{\omega_{d+1}} \binom{n}{d} \int_{\Se} \int_{(v^\perp\cap\hspace{1pt}\Se)^d} \mu(v^+ \cap\Se)^{n-d}\nonumber\\
&\hspace{5mm}\times\;\eta_\Delta({\rm conv}_s\{u_1,\dots,u_d\})\nabla_d(u_1,\dots,u_d)\,\sigma_{d-1}^d(\D(u_1,\dots,u_d))\,\sigma(\D v)+O(c^{n})\nonumber\\ 
&=\left(\frac{2}{\omega_{d+1}}\right)^{d} \binom{n}{d}\ \int_{\Se} \int_{(v^\perp\cap\hspace{1pt}\Se)^d} \left(1-\mu(v^-\cap\Se)\right)^{n-d}\nonumber\\
&\hspace{5mm}\times\;\eta_\Delta({\rm conv}_s\{u_1,\dots,u_d\})\nabla_d(u_1,\dots,u_d)\,
\sigma_{d-1}^d(\D(u_1,\dots,u_d))\,\sigma(\D v)+O(c^{n}).\label{eq2}
\end{align}

Let $v\in \Se\setminus(\{e\}\cup e^\perp)$ be fixed for the moment, with $\alpha:=\alpha(v):=\angle (v,e)\in (0,\frac{\pi}{2})$. We choose $\bar e\in \dS$  such that 
$v=(\cos\alpha)e- (\sin\alpha) \bar e$. 
For $x\in e^\perp\cap \mathbb{S}^+_{\bar e}$, let $x_0$  be the unique vector in the intersection of $v^\perp\cap \Se$ and the geodesic arc connecting $e$ and $x$,  and let $\gamma(x):=\angle(x,x_0)$ denote the angle enclosed by $x$ and $x_0 $.  For a spherical  polytope $U\subset v^\perp\cap\Se$, we then have
$$
\eta_\Delta(U)=\int_{\Pi(U)}\int_0^{\gamma(x)}\cos^{d-1}\gamma\,\D\gamma\,\sigma_{d-1}(\D x).
$$
Next we derive a first order approximation of $\eta_\Delta(U)$ in terms of the angle $\alpha$. For this, we start with deriving a first order approximation of $\gamma(x)$ in terms of $\alpha$. Clearly, 
$$
x-\frac{\langle x,v\rangle }{\langle e,v\rangle}e\in\text{pos}\{e,x\}\cap v^\perp= \text{pos}\{x_0\},
$$
since $\langle x,v\rangle \le 0$ and $\langle e,v\rangle>0$. 
Let $\langle x,\bar e\rangle=:\cos\beta_x$. From
$$
\left\|x-\frac{\langle x,v\rangle }{\langle e,v\rangle}e\right\|^2=1+\frac{\langle x,v\rangle^2}{\cos^2\alpha}
$$
we get
\begin{equation}\label{cosgamma}
\cos\gamma(x)=\left\|x-\frac{\langle x,v\rangle }{\langle e,v\rangle}e\right\|^{-1}=\frac{\cos\alpha }{\sqrt{\cos^2\alpha+\sin^2\alpha\cos^2\beta_x}}
=\frac{\cos\alpha}{\sqrt{1-\sin^2\alpha\sin^2\beta_x}}
\end{equation}
and therefore
\begin{equation}\label{eq1}
\sin\gamma(x) =\frac{\sin\alpha\cos\beta_x}{\sqrt{1-\sin^2\alpha\sin^2\beta_x}}. 
\end{equation}
From (\ref{eq1}) we infer that
$$ \sin \gamma(x) =\left(\alpha+O(\alpha^3)\right)\left(1+O(\alpha^2)\right)\cos\beta_x=\alpha\cos\beta_x+O(\alpha^3),$$
and hence
\begin{equation}\label{eqgamma} 
\gamma(x)=\alpha\langle x,\bar e\rangle+O(\alpha^3).
\end{equation}
Now, the substitution $\gamma=\gamma(x)s$ and \eqref{eqgamma} yield 
\begin{align}
\eta_\Delta(U)&=\int_{\Pi(U)}\int_0^{1}\gamma(x)\cos^{d-1}(\gamma(x)s)\, \D s\, \sigma_{d-1}(\D x)\nonumber\\
&=\int_{\Pi(U)}\int_0^{1}\left(\langle x,\bar e\rangle \alpha +O(\alpha^3)\right)\cos^{d-1}\left((\langle x,\bar e\rangle \alpha +O(\alpha^3))s\right)\, \D s\, \sigma_{d-1}(\D x)\nonumber\\
&=\int_{\Pi(U)}\int_0^1\left(\langle x,\bar e\rangle \alpha+O(\alpha^3)\right)\left(1+O(\alpha^2)\right)\,\D s\, \sigma_{d-1}(\D x)\nonumber\\
&=\int_{\Pi(U)} \left(\langle x,\bar e\rangle \alpha+O(\alpha^3)\right) \, \sigma_{d-1}(\D x)\nonumber\\
&=\alpha\int_{\Pi(U)}  \langle x,\bar e\rangle  \, \sigma_{d-1}(\D x) +O(\alpha^3).\label{eq3}
\end{align}
We combine \eqref{eq2} and \eqref{eq3}. Writing $U:= {\rm conv}_s\{u_1,\dots,u_d\}$ in the following integrals, we obtain
\begin{align*}
&{\mathbb E}\,\sigma(\Se\setminus P_n)\\
&= \binom{n}{d}\left(\frac{2}{\omega_{d+1}}\right)^{d}\int_{\Se}\int_{(v^\perp\cap \Se)^d}\left(1-\frac{\alpha(v)}{\pi}\right)^{n-d}\nabla_d(u_1,\ldots,u_d)\\
&\hspace{5mm}\times \left(\alpha(v) 
\int_{\Pi(U)}\langle x,\bar e\rangle \, \sigma_{d-1}(\D x)+O(\alpha(v)^3)\right)\, \sigma^{d}_{d-1}(\D(u_1,\ldots,u_d))\, \sigma(\D v)+O(c^{n})\\
&= \binom{n}{d}\left(\frac{2}{\omega_{d+1}}\right)^{d}\int_{\Se}\left(1-\frac{\alpha(v)}{\pi}\right)^{n-d} \left(\alpha(v)F(v)+O(\alpha(v)^3)\right)\, \sigma(\D v)+O(c^{n}),
\end{align*}
where 
$$
F(v):=\int_{(v^\perp\cap \Se)^d}\nabla_d(u_1,\ldots,u_d)\int_{\Pi(U)}\langle x,\bar e\rangle \, \sigma_{d-1}(\D x)\,\sigma^{d}_{d-1}(\D(u_1,\ldots,u_d)).
$$
By
\begin{equation}\label{boldeff}
C(d):=\int_{(e^\perp\cap \mathbb{S}^+_{\bar e})^d}\nabla_d(u_1,\ldots,u_d)\int_{U}\langle x,\bar e\rangle \, \sigma_{d-1}(\D x)\,\sigma^{d}_{d-1}(\D(u_1,\ldots,u_d))
\end{equation}
we define a numerical constant which depends merely on the dimension and is independent of the choice of unit vectors  $e,\bar e$ with $e\perp \bar e$. We claim the following.

\vspace{2mm}

\noindent{\em Proposition.} If $\alpha < \pi/4$, then
\begin{equation}\label{eqapprox}
F(v)=C(d)+O(\alpha).
\end{equation}

\vspace{2mm}

To verify the Proposition, let $g:e^\perp\cap \mathbb{S}^+_{\bar e}\to v^\perp\cap \mathbb{S}^+_{e}$ be the mapping defined by
$$g(x):=x_0=(\cos \gamma(x)) x+(\sin \gamma(x)) e.$$ 
Then, for $x,y\in e^\perp\cap \mathbb{S}^+_{\bar e}$ we have
\begin{equation}\label{g1}
g(x)-g(y)=x-y+g(x,y)
\end{equation}
with
$$
g(x,y):=(\cos\gamma(x)-\cos\gamma(y))x+(\cos\gamma(y)-1)(x-y)+(\sin\gamma(x)-\sin\gamma(y))e,
$$
and thus
$$
\|g(x,y)\|\le |\cos \gamma(x)-\cos \gamma(y)|+|\cos \gamma(y)-1|\|x-y\|+|\sin \gamma(x)-\sin \gamma(y)|.
$$
First, from \eqref{eqgamma} we deduce that
$$
|\cos \gamma(y)-1|\le\frac{1}{2}\gamma(y)^2\le O(\alpha^2).
$$
Second, with
$$ A_x:= 1-\sin^2\alpha \sin^2\beta_x\ge \cos^2\alpha$$
we have
\begin{align*}
\left|\frac{1}{\sqrt{A_x}} -\frac{1}{\sqrt{A_y}}\right| & =  \left| \frac{A_x -A_y}{\sqrt{A_x}+\sqrt{A_y}}\right| \frac{1}{\sqrt{A_xA_y}}\\
&\le \frac{1}{2\cos^3\alpha}|A_x-A_y|= \frac{1}{2\cos\alpha^3}\sin^2\alpha\left|\sin^2\beta_x-\sin^2\beta_y\right|\\
&= \frac{\tan^2\alpha}{2\cos\alpha}\left|\langle x,\overline e\rangle^2 - \langle y,\overline e\rangle^2\right|\le  \frac{\tan^2\alpha}{\cos\alpha}\|x-y\|.
\end{align*}
Therefore, from \eqref{cosgamma} we obtain
$$|\cos\gamma(x)-\cos\gamma(y)|=\cos\alpha\left|\frac{1}{\sqrt{A_x}} -\frac{1}{\sqrt{A_y}}\right| \le (\tan^2\alpha)\|x-y\|.$$
Third,  from \eqref{eq1} we deduce that
\begin{align*}
|\sin \gamma(x)-\sin \gamma(y)| &= \sin\alpha \left| \frac{\langle x,\bar e \rangle}{\sqrt{A_x}} -\frac{\langle y,\bar e\rangle} {\sqrt{A_y}} \right|= \sin\alpha\left| \frac{\langle x-y,\bar e\rangle}{\sqrt{A_x}} +\langle y,\overline e\rangle \left(\frac{1}{\sqrt{A_x}} - \frac{1}{\sqrt{A_y}}\right)\right|\\
&\le (\tan\alpha)\|x-y\|+ \sin\alpha\left|\frac{1}{\sqrt{A_x}} - \frac{1}{\sqrt{A_y}}\right|\\
&\le (\tan\alpha)\|x-y\| + (\tan^3\alpha)\|x-y\|.
\end{align*}
Together this gives
\begin{equation}\label{g2}
\|g(x,y)\|\le \left(\tan\alpha+\tan^2\alpha+ \tan^3\alpha+O(\alpha^2)\right)\|x-y\|\le O(\alpha)\|x-y\|,
\end{equation}
since $\alpha < \pi/4$. 

Now let $x\in e^\perp\cap \mathbb{S}^+_{\bar e}$ with $\langle x,\bar e\rangle >0$ and let $w$ be a unit tangent vector of 
$e^\perp\cap \mathbb{S}^+_{\bar e}$ at $x$. Then \eqref{g1} and \eqref{g2} imply that the derivative $\partial_wg(x)$ of $g$ at $x$ in direction $w$ satisfies
$$
\partial_wg(x)=w+O(\alpha),
$$
and therefore, for the Jacobian $J g(x)$ of $g$ at $x$ we obtain  
$$
Jg(x)=1+O(\alpha).
$$
We put $ \overline v:=(\cos\alpha) \bar e+(\sin\alpha) e$. Recalling that $g(x)=(\cos\gamma(x)) x+(\sin\gamma(x)) e$, we get
$$
|\langle x,\bar e\rangle -\langle g(x),\overline v\rangle|=|(1-\cos\alpha\cos\gamma(x))\langle x,\bar e\rangle -\sin\alpha\sin\gamma(x)|\le O(\alpha^2),
$$
and hence, with $ U \subset v^\perp \cap \Se $, 
\begin{align*}
&\left|\int_{\Pi(U)}\langle x,\bar e\rangle\, \sigma_{d-1}(\D x)-\int_{U}\langle z,\overline v\rangle\, \sigma_{ d-1}(\D z)\right|\\
&=\left|\int_{\Pi(U)}\langle x,\bar e\rangle\, \sigma_{d-1}(\D x)-\int_{\Pi(U)}\langle g(x),\overline v\rangle Jg(x)\, \sigma_{ d-1}(\D x)\right|\\
&\le \int_{\Pi(U)}\left|\langle x,\bar e\rangle-\langle g(x),\overline v\rangle\right|\, \sigma_{d-1}(\D x)+O(\alpha)\\
&\le O(\alpha).
\end{align*}
For $0\le \alpha<\pi/4$, we thus get
$$
F(v)=\int_{(v^\perp\cap \Se)^d}\nabla_d(u_1,\ldots,u_d)\int_{U}\langle z, \overline v\rangle \, \sigma_{d-1}(\D z)\,\sigma^{d}_{d-1}(\D(u_1,\ldots,u_d))+O(\alpha).
$$
Applying the rotation which fixes $e^\perp\cap v^\perp$ and maps $e^\perp\cap \mathbb{S}^+_{\bar e}$ to $v^\perp\cap \Se=v^\perp\cap \mathbb{S}^+_{\overline v}$ and $\bar e$ to $\overline v$, we see that the double integral on the right-hand side is equal to $C(d)$, which proves the proposition.

Recalling  \eqref{negli} with $r>\pi/4$ and using Lemma \ref{L2} again, we thus finally get
\begin{align*}
{\mathbb E}\,\sigma(\Se\setminus P_n)&= \binom{n}{d}\left(\frac{2}{\omega_{d+1}}\right)^{d}C(d)
\int_{\Se}\left(1-\frac{\alpha(v)}{\pi}\right)^{n-d}\left(\alpha(v)+O(\alpha(v)^2)\right)\, \sigma(\D v)+O\left(c^n\right)\\
&= \binom{n}{d}\left(\frac{2}{\omega_{d+1}}\right)^{d}\,C(d)\,\omega_d\,
\int_0^{{\pi}}\left(1-\frac{\alpha}{\pi}\right)^{n-d}\left(\sin^d\alpha+O(\alpha^{d+1})\right)\,\D\alpha+O\left(c^n\right)\\
&= \pi^{d+1}\left(\frac{2}{\omega_{d+1}}\right)^{d}\omega_d\,C(d) \, {n^{-1}}+O\left( {n^{-2}}\right)
\end{align*}
and thus (\ref{7.1}).

The expectation of the vertex number is related to that of the volume by the spherical counterpart of Efron's identity, namely
\begin{equation}\label{7.3} 
1-\frac{\E f_0(P_{n+1})}{n+1} =\frac{2}{\omega_{d+1}} \E \sigma(P_n).
\end{equation}
For the reader's convenience, we recall the short proof. Let $X_1,\dots,X_{n+1}$ be independent uniform random points in $\Se$. Define the random variable $N$ as the number of points among $X_1,\dots,X_{n+1}$ that are contained in the spherical convex hull of the others. Then $N= n+1-f_0(P_{n+1})$ and hence $\E N=n+1-\E f_0(P_{n+1})$. If $p$ denotes the probability that $X_1\in{\rm conv}_s\{X_2,\dots,X_{n+1}\}$, then $p=\E \sigma(P_n)/\sigma(\Se)$ and $\E N= (n+1)p$. This gives (\ref{7.3}). Formula (\ref{7.3}) together with (\ref{7.1}) yields (\ref{7.2}).
\end{proof}

Using that $f_0(P_n)=f_1(P_n)$ for $d=2$ and $f_0(P_n)=f_2(P_n)/2+2$ a.s. for $d=3$ (since $P_n$ is almost surely simplicial and the Euler relation holds), we can obtain the explicit values of $C(2)$ and $C(3)$ by comparing (\ref{3.2}) and (\ref{7.2}).

We remark that Theorems \ref{th:LDI-upperbound}--\ref{th:LDI-lowerboundopt} below, combined with Lemma \ref{le:hausdorff}, yield immediate bounds for the almost sure asymptotic behaviour of the missed volume $\sigma(\Se\setminus P_n)$.

\section{Hausdorff Distance}\label{sec8}

The general assumption in this section is again that $X_1,\dots,X_n$ are independent random points in $\Se$, each with distribution $\mu$, and that $P_n$ is their spherical convex hull. In the following, we consider the behaviour of the Hausdorff distance $\delta_s(P_n,\Se)$ as $n\to \infty$. 

Let $\alpha\in(0,\pi/2)$ and $v\in\Se$. If $\langle v,e\rangle =\cos\alpha$, the set $\Se\cap v^-$ is called an $\alpha$-{\em wedge}.

\begin{Lemma}\label{L8.1}
Let $K\subset \Se$ be a nonempty compact convex set, let $\beta\in (0,\pi/2)$. Then $\delta_s(K,\Se)\ge\beta$ if and only if there is a $\beta$-wedge whose interior does not meet $K$.
\end{Lemma}

\begin{proof}
For $K$ as given, it is easy to see that
\begin{equation}\label{8.1}
\delta_s(K,\Se)= \max_{y\in \dS} \min_{x\in K} d_s(x,y).
\end{equation}
Let $\delta_s(K,\Se)=:\alpha\ge\beta$. If $e\notin K$, then $e$ and $K$ can be separated by a great subsphere and the assertion of the lemma follows since $\beta\in (0,\pi/2)$. Hence we can assume that $e\in K$. There are points $x_0\in K$ and $y_0\in \dS$  such that 
\begin{equation}\label{eq:Hd}
\delta_s(K,\Se) =d_s(x_0,y_0)=\alpha.
\end{equation}
Then $x_0$ is the point in $K$ closest to $y_0$. Therefore, $K\cap {\rm int}\,B(y_0,\alpha)=\emptyset$. The disjoint convex sets $K$ and ${\rm int}\,B(y_0,\alpha)$ can be separated by a great subsphere, hence there is a vector $v \in \Se$ such that 
\begin{equation}\label{28a}
K\subset v^+\cap\Se = \{x\in \Se: \langle v,x\rangle \ge 0\}
\end{equation}
and $B(y_0,\alpha)\subset v^-\cap\Se$. 
In particular, $x_0\in v^\perp$. From $e\in K$ and \eqref{eq:Hd} we conclude that $B(e,\frac{\pi}{2}-\alpha)\subset K$,  therefore 
 $x_0=(\cos\alpha)y_0+(\sin\alpha)e$, and then $v= (\cos\alpha)e-(\sin\alpha)y_0$. 

This shows that $K$ does not meet the interior of the $\alpha$-wedge $v^-\cap \Se$. A fortiori, there is a $\beta$-wedge with the same property. This proves one direction of the assertion, and the other direction is obvious.
\end{proof} 

Now we show that the Hausdorff distance is closely related to the missed volume.

\begin{Lemma}\label{le:hausdorff}
For $K$ as in the previous lemma,
$$ \frac {\omega_{d+1} }{2 \pi }\,  \delta_s(K,\Se) \le \sigma(\Se\setminus K) \le \omega_{d}\, \delta_s(K,\Se) .$$
\end{Lemma}

\begin{proof}
Let 
$\delta_s(K,\Se)=:\alpha$. As shown in the previous proof, there is a vector $v\in\Se$ with $\langle v,e\rangle=\cos\alpha$ and such that (\ref{28a}) holds. Therefore,
$$ \frac {\omega_{d+1} }{2 \pi } \alpha = \sigma(v^-\cap\Se)\le \sigma(\Se\setminus K) . $$
This proves the lower bound. For the upper bound, observe that for $y \in \Se \setminus K$ equation (\ref{eq:Hd}) together with (\ref{8.1}) implies that
$$\langle y, e \rangle \le \sup_{y \in \Se \setminus K} \langle y, e \rangle  = \langle x_0, e \rangle = \sin \alpha. $$
Hence, 
$$ \sigma(\Se\setminus K) \le \sigma\left( \{y \in \Se:\, \langle y,e \rangle \le \sin \alpha \} \right) \leq  \omega_d\alpha, $$
as stated.
\end{proof}

Taking expectations and observing Theorem \ref{T7.1}, we obtain as an immediate consequence that the order of $\E \delta_s(P_n, \Se)$ is $1/n$.

\begin{Theorem}\label{T8.1} There are constants $ c_1, c_2 $, depending only on the dimension, such that
$$  {c_1}\, {n^{-1}} \le  \E \delta_s(P_n,\Se) \le {c_2}\, {n^{-1}}.$$
\end{Theorem}

Now we derive almost sure upper and lower bounds. In the next two theorems, we may assume that $X_1,X_2,\dots$ is an independent sequence of random points in $\Se$, each with distribution $\mu$, and that $P_n$ is the spherical convex hull of the first $n$ points of this sequence.


\begin{Theorem}\label{th:LDI-upperbound} 
There is a constant $C$ depending only on the dimension such that
\begin{equation}\label{eq:LDI-upperbound}
\Pb \left(\delta_s(P_n,\Se) \le C\,\frac {\ln n}{n} \mbox{ \rm for almost all }n\right)=1.
\end{equation}
\end{Theorem}

\begin{proof} 
We give a short proof based on the following interesting result of Vu \cite[Lemma 4.2]{VVu05} (a direct proof of (\ref{eq:LDI-upperbound}) not using that lemma can also be given, but is slightly longer). Let $\psi$ be a probability measure on $\mathbb{R}^d$. Choose independent random points $Y_1,\ldots,Y_n$ with distribution $\psi$ and let $K_n:=\mbox{conv} \{Y_1,\ldots,Y_n\}$. For $t>0$, the $t$-{\sl floating body}, $K(t)$, is defined as the closure of the set of points $x \in \mathbb{R}^d$ that are not contained in any closed halfspace $H$ with $\psi(H)\le t$. It is not hard to see that $K(t)$ is a convex set. The result from \cite{VVu05} that we need says that there are positive constants $b_1,b_2$, depending only on the dimension, such that for sufficiently large $n$ and for any $t\ge b_1\frac {\ln n}{n}$,
\begin{align}\label{VanVu}
\Pb\left(K(t)\not\subset K_n\right)\le\exp\{-b_2tn\}.
\end{align}

Now recall from Section \ref{sec3} that the image measure of $\mu$, restricted to the interior of $\Se$, under the gnomonic projection $h$ that maps $v$ to $h(v)=\langle e,v\rangle^{-1}v-e$, is a probability measure $\psi$ on $e^\perp$, the latter identified with ${\mathbb R}^d$. Clearly, $h$ is one-to-one between the interior of $\Se$ and $e^\perp$. The random spherical polytope $P_n$ is mapped by $h$ to the random polytope $K_n$ (chosen according to $\psi$) and conversely.

Assume $z \in \Se$ and $\langle e,z\rangle =\cos \alpha$ with $\alpha \in (0,\pi/2)$. The halfspace $z^-=\{x\in \mathbb{R}^{d+1}: \langle z,x\rangle \le 0\}$ intersects $\Se$ in a set of $\mu$-measure $\alpha/\pi$, and $h(z^-\cap \Se)$ is a halfspace in $e^\perp$ whose $\psi$-measure is $\alpha/\pi$. Also conversely, for every halfspace $H$ in $e^\perp$ with $\psi(H)=\alpha/\pi$, the pre-image $h^{-1}(H)$ is of the form  $z^-\cap \Se$ for some $z \in \Se$ with $\langle e,z\rangle=\cos \alpha$. It follows that the $\alpha/\pi$-floating body of the measure $\psi$ is the $h$-image of $B(e,\pi/2-\alpha)$. Vu's lemma applies in $e^\perp$ and, via the inverse of the gnomonic map, also in $\Se$. There it says that there are constants $b_1,b_2>0$ such that for large $n$ and any $\alpha \ge b_1\frac {\ln n}n$,
\[
\Pb\left(B(e,\pi/2-\alpha)\not\subset P_n\right)\le\exp\{-b_2\alpha n\}.
\]
We now choose $C\ge b_1$ so that $Cb_2\ge 2$ and set $r=\pi/2-C\frac {\ln n}n$. Then it follows that
$$
\mathbb{P}(B(e,r)\not\subset P_n)\le \exp\{-Cb_2 \ln n\}\le {n^{-2}}.
$$
As $B(e,r)\not\subset P_n$ is equivalent to $\delta_s(P_n,\Se) > C\frac {\ln n}{n}$, the last inequality  implies the theorem via the Borel--Cantelli lemma.
\end{proof}

In the other direction, we can only show the following.

\begin{Theorem}\label{th:LDI-lowerbound} 
For any $\gamma>2$ we have
\begin{equation}\label{eq:LDI-lowerbound}
\Pb\left( \delta_s(P_n,\Se) \ge  n^{-\gamma} \enspace \mbox{ \rm for almost all }n\right )=1.
\end{equation}
\end{Theorem}

\begin{proof}
Let $0<\eps<\pi/2$ be given. Choose $y\in\dS$ and let $v:= (\cos\eps)e+(\sin\eps)y$ and $W:= \Se\cap v^-$. Then
\begin{align*}
\Pb ( \delta_s(P_n,\Se) \le\eps ) 
&\le \Pb (W\cap\{X_1,\dots,X_n\}\not=\emptyset) \\
&= 1-\Pb (X_i\notin W \text{ for }i=1,\dots,n ) =1-\prod_{i=1}^n \Pb (X_i\notin W) \\
&= 1- \prod_{i=1}^n(1-\Pb(X_i\in W))= 1-(1-\mu(W))^n \\
&= 1-\left(1-\frac{\eps}{\pi}\right)^n.
\end{align*}
Now let
$$ \eps_n:={n^{-(2+\eta)}}\quad\mbox{with } \eta>0.$$
Then
$$ \Pb (\delta_s(P_n,\Se) \le\eps_n ) \le \frac{1}{\pi} \, {n^{-(1+\eta)}} +O(n^{-2}).$$
It follows that
$$ \sum_{n=1}^\infty \Pb (\delta_s(P_n,\Se) \le\eps_n) <\infty.$$
Hence, the Borel--Cantelli lemma gives
$$  \Pb (\delta_s(P_n,\Se) \le\eps_n \text{ for infinitely many }n ) =0$$
and, therefore,
$$  \Pb\left(\delta_s(P_n,\Se) > {n^{-(2+\eta)}}\text{ \rm for almost all }n\right ) = 1,$$
which is (\ref{eq:LDI-lowerbound}).
\end{proof}

Under stronger independence assumptions, we can give counterparts to the preceding two theorems. 

\noindent{\bf Assumption $(*)$.} For each $n\in{\mathbb N}$, $X_1^{(n)},\dots,X_n^{(n)}$ are independent random points in $\Se$, each with distribution $\mu$, and $P_n$ is their spherical convex hull. Writing $A_n$ for the $n$-tuple $(X_1^{(n)}, \dots, X_n^{(n)})$, the sequence $A_1,A_2,\dots$ is independent.

\begin{Theorem}\label{th:LDI-upperboundopt} 
Under Assumption $(*)$, 
\begin{equation*}
\Pb\left(\delta_s(P_n,\Se) \ge c\, \frac {\ln n}{n}  \mbox{ \rm for infinitely many }n\right)=1,
\end{equation*}
with $c=\omega_{d+1}/5\omega_d$.
\end{Theorem}

\begin{proof}
First observe that, for arbitrary $c>0$, 
\begin{eqnarray*}
& & \lim_{n \to \infty} \frac{n}{\ln n}  \, \mu \left(\left\{ y\in \Se:\, \langle y, e \rangle \leq c\, \frac {\ln n}{n} \right\} \right) \\
&& = \frac{2 }{\omega_{d+1}} \lim_{n \to \infty} \frac{n}{\ln n} \, \sigma \left(\left\{ y\in\Se:\, \langle y, e \rangle \le c\, \frac {\ln n}{n} \right\} \right) \\
&& = \frac{2 }{\omega_{d+1}} \lim_{n \to \infty} \frac{n}{\ln n} \, \omega_d \, c\, \frac {\ln n}{n} =
\frac{2 \omega_d } {\omega_{d+1}} \, c.
\end{eqnarray*}
Thus, if $c=\omega_{d+1}/5\omega_d$, then 
$$
\mu \left(\left\{ y\in\Se:\, \langle y, e \rangle \leq c\, \frac {\ln n}{n} \right\} \right)  
\le \frac 12\, \frac {\ln n }n 
$$
for all sufficiently large $n$. For these $n$ we get
\begin{align*}
\Pb \left(\delta_s(P_n,\Se) \ge c\, \frac {\ln n}{n} \right)
& \ge  \Pb \left(X_1^{(n)}, \dots, X_n^{(n)}  \in \left\{ y\in\Se: \langle y, e \rangle \geq c\, \frac {\ln n}{n} \right\}\right)\\ 
&= \mu \left(\left\{ y\in\Se: \langle y, e \rangle \ge c\, \frac {\ln n}{n} \right\}\right)^n\\ 
& =  \left( 1-  \mu \left(\left\{ y\in\Se: \langle y, e \rangle \leq c\, \frac {\ln n}{n} \right\}\right) \right)^n \\ 
& \ge  \left( 1- \frac 12\, \frac {\ln n }n \right)^n
\\ & \geq 
e^{-\ln n} = n^{-1},
\end{align*}
because $1- \frac x2 \geq e^{-x}$ for $x \in [0,1]$. 
This yields
$$ \sum_{n=1}^\infty \Pb \, \left(\delta_s(P_n,\Se) \ge c\, \frac {\ln n}{n} \right) = \infty, 
$$
and the Borel-Cantelli lemma yields the assertion.
\end{proof}

Observe that Theorems \ref{th:LDI-upperbound} and \ref{th:LDI-upperboundopt} imply that, under Assumption $(*)$,
$$
\Pb \left(c \leq  \limsup_{n \to \infty} \frac {n}{\ln n}\,\delta_s(P_n,\Se)  \leq C \right)=1 , $$
and, in particular,
$$ \Pb \left(\limsup_{n \to \infty} n\,\delta_s(P_n,\Se) =\infty\right)= 1 .$$

\vspace{2mm}

Our counterpart to Theorem \ref{th:LDI-lowerbound} is of a slightly different order. 

\begin{Theorem}\label{th:LDI-lowerboundopt} 
Under Assumption $(*)$, there is a number $0<\eps<1$, depending only on the dimension, such that
\begin{equation*}
\Pb\left(\delta_s(P_n,\Se) \le n^{-(1 + \eps)}\enspace \mbox{\rm for infinitely many } n\right)=1.
\end{equation*}
\end{Theorem}

\begin{proof}
We choose a saturated sequence of pairwise disjoint balls $B_1, \dots, B_m$ of radius $1/4$ in $\partial \Se$ (where $m$ depends only on the dimension) and define $C_k$ as the intersection of all $\alpha$-wedges containing $B_k$, $k=1,\dots,m$, where $\alpha$ will be specified soon. Since the system of the balls is saturated, each halfsphere of $\partial \Se$ contains at least one of the balls $B_k$ and hence  each $\alpha$-wedge contains at least one of the sets $C_k$. The sets $C_k$ are pairwise disjoint, and
$$ \mu(C_k) \ge c_1\alpha$$
with a suitable constant $c_1>0$ depending only on the dimension.

If each set $C_k$ contains one of the random points $X_1^{(n)}, \dots , X_n^{(n)}$, then $\delta_s(P_n,\Se) \le \alpha$ by Lemma \ref{L8.1}, thus
$$ \Pb\left( \delta_s(P_n,\Se) \le  \alpha \right)  \ge \Pb\left( \left\{X_1^{(n)},\ldots,X_n^{(n)}\right\}\cap C_k\neq\emptyset \;\mbox{ for }k=1,\ldots,m\right).$$
Define $C_{m+1}:= \Se\setminus\bigcup_{k=1}^m C_k$. We have $\left\{X_1^{(n)},\ldots,X_n^{(n)}\right\}\cap C_k\neq\emptyset$ for $k=1,\ldots,m$ if and only if there are  sets $I_1,\dots,I_{m+1}\subset \{1,\ldots,n\}$ which partition $\{1,\ldots,n\}$ and are such that $|I_i|\ge 1$ for $i=1,\dots,m$ and $X_r^{(n)}\in C_k$ for $r\in I_k$, $k=1,\dots,m+1$. Hence, using the fact that all the sets $C_k$ have the same spherical volume, we get, for $n\ge m$,
\begin{eqnarray*}
&& \Pb\left( \left\{X_1^{(n)},\ldots,X_n^{(n)}\right\}\cap C_k\neq\emptyset \;\mbox{ for }k=1,\ldots,m\right)\\
&&= \sum_{l_1,\dots,l_m\ge 1,\; l_{m+1}\ge 0}^{n} \binom{n}{l_1,\dots,l_m,l_{m+1}} \mu(C_1)^{l_1+\cdots+l_m} \left(1-\sum_{k=1}^m \mu(C_k)\right)^{l_{m+1}}\\
&& =\binom{n}{1,\dots,1,n-m}\mu(C_1)^{m}\left(1-m\mu(C_1)\right)^{n-m}+\mbox{nonnegative terms}\\
&& \ge \frac{n!}{(n-m)!}c_1^m \alpha^m \left(1-mc_1\alpha\right)^{n-m}=:A.
\end{eqnarray*}
If we now choose $\alpha= n^{-(1+\eps)}$ with $\eps:= m^{-1}$, then
$$ A= \frac{n!}{(n-m)!n^m}\, c_1^m {n^{-1}} \left[\left(1-\frac{c_1m}{n^{1+\eps}}\right)^{n^{1+\eps}} \right]^{\frac{n-m}{n^{1+\eps}}} 
\ge c_2(m)c_1^mc_3(m){n^{-1}},$$
where the constants $c_1,c_2,c_3$ are positive. We conclude that
$$ \sum_{n=m}^\infty \Pb\left(\delta_s(P_n,\Se)\le n^{-(1+\eps)}\right) =\infty,$$
and the Borell--Cantelli lemma yields the assertion.
\end{proof}

Observe that Theorem \ref{th:LDI-lowerboundopt} implies that, under Assumption $(*)$,
$$ \Pb \left(\liminf_{n \to \infty} n\delta_s(P_n,\Se) =0\right)= 1 . $$

\section{Concluding Remarks}\label{sec9}

If we define $S(P)=:F_{d-1}(P)$, $U_1(P)=:F_1(P)$, $\sigma(P)=:F_d(P)$ for a spherical polytope $P\in\Se$, then we can write the results of Theorems \ref{T2}, \ref{T3}, \ref{T7.1} in the form
\begin{equation}\label{9.1}
{\mathbb E}\left(F_i(\Se)-F_i(P_n)\right) =c_i(d)n^{-(d+1-i)} + O\left(n^{-(d+2-i)}\right),\qquad i=1,d-1,d,
\end{equation}
where $c_i(d)$ is a constant which is explicitly known for $i=1$ and $i=d-1$, but not for $i=d$. The dependence of the asymptotic expansion in (\ref{9.1}) on the number $i$ is a phenomenon that does not occur for the analogous problem for quermassintegrals (intrinsic volumes) of random polytopes in smooth convex bodies in Euclidean space; compare \cite{Bar92} and \cite{Rei04}. It would be interesting to extend (\ref{9.1}) to the remaining functionals $F_i$, $i=2,\dots,d-2$. These could either be the spherical quermassintegrals, defined by \cite[(6.62)]{SW08}, or the spherical intrinsic volumes, defined in \cite[p. 256]{SW08}. In spherical space, these are two different series of functionals, though related by linear relations; see \cite[(6.63)]{SW08}.

Also the asymptotic behaviour of the expected number of $k$-faces of $P_n$, which we have determined for $k=d-1$ in Theorem \ref{T1} (and hence also for $k=d-2$) and for $k=0$ in Theorem \ref{T7.1} (with an unknown constant), would be interesting to know in the remaining cases. 

Finally, Theorem \ref{th:LDI-upperbound} should be compared with results about the Hausdorff distance in Euclidean spaces, which are in a similar spirit, though distinctly different in the orders; we refer to \cite{DW96} and, for circumscribed random polytopes, to \cite{HS14}. It would be interesting to know whether the lower bound (\ref{eq:LDI-lowerbound}) can be improved.

\noindent Authors' addresses:

\noindent Imre B\'{a}r\'{a}ny\\
R\'enyi Institute of Mathematics \\
Hungarian Academy of Sciences \\
H-1364 Budapest, Hungary\\
e-mail: barany.imre@renyi.mta.hu

\noindent and \\
Department of Mathematics \\
University College London \\
Gower Street, London WC1E 6BT, England

\noindent Daniel Hug\\
Karlsruhe Institute of Technology \\
Department of Mathematics \\
D-76128 Karlsruhe, Germany\\
e-mail: daniel.hug@kit.edu

\noindent Matthias Reitzner\\
Department of Mathematics\\
University of Osnabr\"uck\\
D-49076 Osnabr\"uck, Germany\\
e-mail: matthias.reitzner@uni-osnabrueck.de

\noindent Rolf Schneider\\
Mathematisches Institut\\
Albert-Ludwigs-Universit\"at\\
D-79104 Freiburg i. Br., Germany\\
e-mail: rolf.schneider@math.uni-freiburg.de

\end{document}